\documentclass[12pt]{amsart}

\usepackage{fix-cm}
\usepackage[normalem]{ulem}

\usepackage{lineno}

\usepackage{soul}

\usepackage{graphicx}

\usepackage{amssymb, amsmath, amsthm, hyperref, euscript, color}

\usepackage[dvipsnames]{xcolor}

\usepackage{amscd}

\usepackage{tikz-cd}

\usepackage{bbm}

\usepackage{enumitem}

\usepackage[english]{babel}

\usepackage{mathtools}

\usepackage{makecell}

\usepackage[sorting=nty]{biblatex} 
\numberwithin{equation}{section}

\newtheoremstyle{theor}{6pt plus 1pt minus 1pt}{6pt plus 1pt minus 1pt}{\slshape}{}{\bfseries}{.}{5pt plus 1pt minus 1pt}{}
\newtheoremstyle{def}{6pt plus 1pt minus 1pt}{6pt plus 1pt minus 1pt}{}{}{\bfseries}{.}{5pt plus 1pt minus 1pt}{}
\newtheoremstyle{rmk}{6pt plus 1pt minus 1pt}{6pt plus 1pt minus 1pt}{}{}{\bfseries}{.}{5pt plus 1pt minus 1pt}{}
\newtheoremstyle{claim}{6pt plus 1pt minus 1pt}{6pt plus 1pt minus 1pt}{}{}{\bfseries}{.}{5pt plus 1pt minus 1pt}{}

\theoremstyle{theor}
\newtheorem{newstatement}{newstatement}
\newtheorem{lemma}[newstatement]{Lemma}
\newtheorem{theorem}[newstatement]{Theorem}
\newtheorem*{theorem*}{Theorem 2}

\newtheorem{proposition}[newstatement]{Proposition}

\theoremstyle{def}

\theoremstyle{rmk}
\newtheorem{remark}[newstatement]{Remark}

\newtheorem*{example*}{Example}

\theoremstyle{claim}

\theoremstyle{theor}
\newtheorem{thm}{Theorem}

\expandafter\let\expandafter\oldproof\csname\string\proof\endcsname
\let\oldendproof\endproof
\renewenvironment{proof}[1][\proofname]{%
  \oldproof[\slshape #1]%
}{\oldendproof}
\def\provedboxcontents#1{$\square$}

\makeatletter
\newsavebox\myboxA
\newsavebox\myboxB
\newlength\mylenA

\newcommand*\xoverline[2][0.75]{%
    \sbox{\myboxA}{$\m@th#2$}%
    \setbox\myboxB\null
    \ht\myboxB=\ht\myboxA%
    \dp\myboxB=\dp\myboxA%
    \wd\myboxB=#1\wd\myboxA
    \sbox\myboxB{$\m@th\overline{\copy\myboxB}$}
    \setlength\mylenA{\the\wd\myboxA}
    \addtolength\mylenA{-\the\wd\myboxB}%
    \ifdim\wd\myboxB<\wd\myboxA%
       \rlap{\hskip 0.5\mylenA\usebox\myboxB}{\usebox\myboxA}%
    \else
        \hskip -0.5\mylenA\rlap{\usebox\myboxA}{\hskip 0.5\mylenA\usebox\myboxB}%
    \fi}
\makeatother

\newcommand{\Q}{\mathbb{Q}}

\newcommand{\Z}{\mathbb{Z}}

\newcommand{\C}{\mathbb{C}}

\DeclareMathOperator{\Sp}{Spin}

\DeclareMathOperator{\Sw}{SW}

\makeatletter
\@namedef{subjclassname@2020}{%
  \textup{2020} Mathematics Subject Classification}
\makeatother

\begin{document}

\author{Sebasti\'an M. Camponovo and Rafael Torres}

\title[\tiny EXOTIC DEFINITE FOUR-MANIFOLDS WITH INFINITE FUNDAMENTAL GROUP]{SMOOTH STRUCTURES ON DEFINITE FOUR-MANIFOLDS WITH INFINITE FUNDAMENTAL GROUP}

\address{Dipartimento di Matematica, Informatica e Geoscienze, Università di Trieste, Via Valerio 12/1, 34127
Trieste, Italy}

\email{sebastianmatias.camponovo@phd.units.it}

\address{Scuola Internazionale Superiori di Studi Avanzati (SISSA)\\Via Bonomea 265\\34136\\Trieste\\Italy}

\email{rtorres@sissa.it}

\subjclass[2020]{57R55, 57K41, 57M10}

\begin{abstract}For each odd integer $p > 1$, we construct infinitely many pairwise non-diffeomorphic irreducible smooth structures on a definite 4-manifold with infinite fundamental group whose abelianization is $\Z/2p\Z\times \Z/2\Z$. 
\end{abstract}

\maketitle

\section{Introduction.}\label{Introduction}

 Notorious progress in the quest of understanding smooth structures on four-manifolds with definite intersection form has happened in the past few years. Levine-Lidman-Piccirillo \cite{levine2023new} constructed the first examples of non-diffeomorphic smooth structures on a four-manifold with negative definite intersection form and fundamental group of order two. Stipsicz-Szab\'o \cite{stipsicz2024definite, stipsicz2024definite2} and Baykur-Stipsicz-Szab\'o \cite{baykur2024smooth} produced infinite families of inequivalent irreducible smooth structures in vast collections of homeomorphism classes of four-manifolds with finite cyclic fundamental group, which include examples with Euler characteristic as small as three. Harris-Naylor-Park \cite{harris} constructed an infinite family on a four-manifold with non-cyclic fundamental group $\Z/2\Z\times\Z/2\Z$.

The main result of this brief contributes to such progress.
\begin{thm}\label{mainthm}
    Let $p>1$ be an odd integer. There is a closed smooth oriented four-manifold $X_p$ that satisfies the following properties. 
    \begin{enumerate}
        \item Its fundamental group is infinite with abelianization given by $\text{Ab}(\pi_1(X_p))=\Z/2p\Z\times\Z/2\Z$.
         \item Its intersection form is negative-definite.
        \item Its Euler characteristic and its signature are\begin{center} $\chi(X_p) = 6$ and $\sigma(X_p)=-4$.\end{center}
        \item There are infinitely many pairwise non-diffeomorphic irreducible smooth structures in the homeomorphism class of $X_p$.
    \end{enumerate}
\end{thm}

These are the first examples of inequivalent irreducible smooth structures on four-manifolds with definite intersection form and infinite fundamental group to the best of our knowledge. These groups have exponential growth; see Remark \ref{Remark Growth}. Further generalizations are described in Remark \ref{Remark Extension}. 

Theorem \ref{mainthm} is obtained from the work of Harris-Naylor-Park by using their example as raw material in the cut-and-paste construction mechanism of four-manifolds introduced by Bais-Torres in \cite{bais}; see Section \ref{Section Construction Mechanism} for a description of the latter. Given their paramount relevance to our purposes, the results of Harris-Naylor-Park are summarized in Section \ref{Section Raw Materials}. The examples of Theorem \ref{mainthm} are put together in Section \ref{Section Construction Examples}, while their invariants are computed in Section \ref{top}. Their homeomorphism class is identified in Section \ref{top} as well. Their diffeomorphism classes are discerned in Section \ref{sw} through the use of  their virtual Seiberg-Witten invariants following the strategy in \cite{baykur2024smooth, harris,  stipsicz2024definite, stipsicz2024definite2}. An interesting trait of our examples is that their irreducibility follows from entirely topological arguments; see proof of Proposition \ref{Proposition Diffeomorphism}. The produce of these sections is gathered together into a proof of Theorem \ref{mainthm} in Section \ref{Section Proof Theorem Main}.

Unless otherwise stated, all maps are smooth, all manifolds are closed, oriented and smooth, and all (co)homology groups are taken with coefficients in $\Z$.

\subsection{Acknowledgements} S.M.C. would like to thank Tye Lidman for helpful comments on an earlier version of this manuscript. S.M.C. is a member of GNSAGA – Istituto Nazionale di Alta Matematica ‘Francesco Severi’, Italy.
\section{Assemblage of the four-manifolds of Theorem \ref{mainthm} and their features.}\label{lego}

\subsection{Seiberg-Witten invariants}\label{Section SW}
In this section, we establish the conventions and notations on the Seiberg-Witten invariants that are used in the sequel. The reader is directed towards \cite[Lecture 2]{fintushel2006six} and \cite[Section 2]{harris} for further details; we follow closely the exposition of the latter. The Seiberg-Witten invariant of a four-manifold $X$ with $b_2^+(X) > 1$ is an integer-valued map\begin{equation}\label{Spin Domain}\Sw_X:\Sp^\C(X)\to\Z\end{equation} on the set $\textnormal{Spin}^\C(X)$ of spin$^\C$-structures on $X$. A spin$^\C$-structure $\mathfrak{s}$ on $X$ is denoted by $(X, \mathfrak{s})$. To obtain a map\begin{equation}\label{Cohomology Domain}\Sw_X:H^2(X)\rightarrow \Z,\end{equation}we follow \cite[Section 2]{harris} and consider the set\begin{equation}\label{Set}I_K=\{\mathfrak{s}\in\Sp^\C(X)\;|\;c_1(X,\mathfrak{s})=K\}.\end{equation} for a given $K\in H^2(X)$. One defines\begin{equation}\Sw_X(K)=\sum_{\mathfrak{s}\in I_K}\Sw_X(\mathfrak{s}),\end{equation} where $\Sw_X(K)=0$ if (\ref{Set}) is empty. This yields a function (\ref{Cohomology Domain}) with finite support and a diffeomorphism invariant.

The Seiberg-Witten invariants of $X$ can be encapsulated as a single invariant\begin{equation}\label{Laurent Polynomial}\overline{\Sw}_X=\sum_{K\in \Z[H^2(X)]}\Sw_X(K)\;\text{for }K\in\Z[H^2(X)]\end{equation} in the group ring $\Z[H^2(X)]$.

We denote by\begin{equation}\label{Family Knots}\mathcal{K}:= \{K_m: m\in \Z_{> 0}\}\end{equation}a family of knots $K_m\subset S^3$ with symmetrized Alexander polynomial equal to\begin{equation}
    \Delta_{K_m}(t)=mt-(2m-1)+mt^{-1}\end{equation}for each positive $m > 0$. As mentioned in \cite[p. 443]{harris}, examples of such knots are the alternating twist knots with $2m - 1$ half twists.

\subsection{Feedstock four-manifolds}\label{Section Raw Materials} This section contains a brief description of the work of Harris-Naylor-Park given its relevance in the proof of Theorem \ref{mainthm}. In \cite{harris}, the authors first describe a free orientation-preserving $\Z/2\Z\times \Z/2\Z$-action on an elliptic fibration of the $K3$ surface, which preserves the elliptic fibration structure by mapping fibers to fibers. They perform four Fintushel-Stern's knot surgery \cite{Fintushel1996KnotsLA} on four disjoint smooth tori $\{T_1, T_2, T_3, T_4\}$ in $K3$ that lay in the same orbit of the group action and in such a way that the free $\Z/2\Z\times \Z/2\Z$-action on the complement $K3\setminus \overset{4}{\underset{i = 1}\sqcup} \nu(T_i)$ extends to the new four-manifold\begin{equation}\label{New K3}K3_{K_m}:= (K3\setminus \overset{4}{\underset{i = 1}\sqcup} \nu(T_i)) \cup \mathcal{E}_i.\end{equation} Here, $\mathcal{E}_i = S^ 1\times (S^3\setminus \nu(K_m))$ for $i \in \{1, 2, 3, 4\}$, where $S^3\setminus \nu(K_m)$ is the complement of the tubular neighborhood $\nu(K_m)$ of a given knot $K_m\in \mathcal{K}$. The quotient of $K3_{K_m}$ by this group action produces a four-manifold $Y_{K_m}$ with fundamental group $\pi_1(Y_{K_m}) = \Z/2\Z\times \Z/2\Z$. By varying the choice of knot in the family (\ref{Family Knots}) used in the construction, Harris-Naylor-Park obtained a set $\{Y_{K_m}: K_m\in \mathcal{K}\}$. The main result of \cite{harris} is the following theorem. A smooth torus fiber of an elliptic fibration $K3\rightarrow \mathbb{CP}^2$ is denoted by $T$ and $[0]\in H^2(K3)$ is the trivial class.

\begin{theorem}[Harris-Naylor-Park]\label{bb}

There is a four-manifold $Y$ with the following properties.
\begin{enumerate}
    \item Its fundamental group is $\pi_1(Y)= \Z/2\Z\times\Z/2\Z$.
    \item Its Euler characteristic is $\chi(Y) = 6$, its second Betti number is $b_2(Y)=4$, and its intersection form is negative definite. 
    \item There are infinitely many pairwise non-diffeomorphic smooth structures\begin{equation}\label{Family HNP}\{Y_{K_m}: K_m\in \mathcal{K}\}\end{equation}in the homeomorphism class of $Y$.
    \item There is a four-to-one covering $K3_{K_m}\rightarrow Y_{K_m}$ with Seiberg-Witten invariant\begin{align*}
        \overline{\Sw}_{K3_{K_m}}&=(\Delta_{K_m}(PD(2[T])))^4 =\\ 
        &= (mPD(2[T]) - (2m - 1)[0] + m(-PD(2[T])))^4,\end{align*}where $PD: H_2(K3_{K_m})\rightarrow H^2(K3_{K_m})$ is the Poincar\'e duality homomorphism and the product on the right-hand side is the product in the group ring $\Z[H^2(K3_{K_m})]$.

\end{enumerate}    
\end{theorem}

  The next lemma is implicit in the work of Harris-Naylor-Park. Associated to every order two element in $\pi_1(Y_{K_m})$ we get a degree two cover $K3_{K_m}\rightarrow E_{K_m}$, where $E_{K_m}$ is homeomorphic to the Enriques surface $E_{K_m}$ by a result of Hambleton-Kreck \cite[Theorem C]{kreck1993cancellation}.
    \begin{lemma}\label{Lemma Implicit}There is an infinite set $\{E_{K_m} : K_m\in \mathcal{K}\}$ of pairwise non-diffeomorphic irreducible four-manifolds such that each of its elements is homeomorphic to the Enriques surface and it has a double covering $K3_{K_m}\rightarrow E_{K_m}$.
\end{lemma}

Theorem \ref{bb} and Lemma \ref{Lemma Implicit} can be gathered together in the following diagram of coverings \[\begin{tikzcd}
	& {K3_{K_m}} \\
	{E_{K_m}} \\
	& {Y_{K_m}}
	\arrow["{2:1}"', from=1-2, to=2-1]
	\arrow["{4:1}", from=1-2, to=3-2]
	\arrow["{2:1}"', from=2-1, to=3-2]
\end{tikzcd}\]

The existence of the diagram corresponding to the standard smooth structures follows from the free involution on the $K3$ surface along with the free antiholomorphic involution on the Enriques surface described by Hitchin \cite[p. 440]{hitchin}.

\subsection{Cut-and-paste construction mechanism of four-manifolds}\label{Section Construction Mechanism}The first three steps (\textbf{S.1}, \textbf{S.2} and \textbf{S.3}) in the construction mechanism devised in \cite{bais} are as follows.\begin{enumerate}[label=\textbf{S.\arabic*}]
\item Begin with a pair of four-manifolds $(M_1, M_2)$ that are homeomorphic. Remove the tubular neighborhood $\nu(\alpha_i)$ of a simple loop $\alpha_i\subset M_i$ to produce a compact four-manifold
\begin{equation*}
M_i^0\mathrel{\mathop:}= M_i\setminus \nu(\alpha_i)
\end{equation*}with boundary $\partial(M_i^0) = S^1\times S^2$ for $i = 1, 2$. 

\item Consider the simple loop $$\beta=S^1\times\{pt\}\subset S^1\times S^3$$ whose homotopy class $[\beta]\in \pi_1(S^1\times S^3)=\Z$ generates. Denote by $\beta_p\subset S^1\times S^3$ the simple loop that represents the first homology class $[\beta]^p\in H_1(S^1\times S^3)=\Z$ for $p > 1$ an odd integer. Carve this codimension zero submanifold out of $S^1\times S^3$ to obtain the compact four-manifold\begin{equation}\label{Building Block Sp}S_p  = (S^1\times S^3)\setminus\nu(\beta_p)\end{equation} with boundary $\partial S_p = S^1\times S^2$. A handlebody of $S_p$ is drawn in Figure 1; see Remark \ref{Remark Handlebody}.

\item Assemble new four-manifolds\begin{equation}\label{Manifolds Produced}M_i(p)\mathrel{\mathop:}= M_i^0\cup S_p\end{equation} for $i = 1, 2$.
\end{enumerate}

A priori, the diffeomorphism type of the four-manifolds (\ref{Manifolds Produced}) depends on the $\Z/2\Z$-choice of framing of the simple loops involved \cite[Section 5.1]{gompf19994}. The next theorem guarantees that this is not the case provided $p > 0$ is indeed an odd integer; see Remark \ref{Remark Hypothesis}.

\begin{theorem}[Bais-Torres {\cite[Theorem A]{bais}}]\label{Theorem Extension}Let $p > 0$ be an odd integer. For any diffeomorphism $h: S^1\times S^2\rightarrow S^1\times S^2$, there is a diffeomorphism $H_p: S_p\rightarrow S_p$ such that $H_p|_{\partial} = h$. 
\end{theorem}

A consequence of Theorem \ref{Theorem Extension}, quite useful for our purposes, allows us to pin down the homeomorphism class of the four-manifolds (\ref{Manifolds Produced}) under some mild conditions.

\begin{theorem}[Bais-Torres {\cite[Theorem B]{bais}}]\label{homeo}
    Let $M_1$ and $M_2$ be topological four-manifolds such that there are locally flat simple loops \[\{\alpha_{M_i}\subset M_i: i = 1, 2\}\] for which there is a homeomorphism
\begin{equation*}
M_1\setminus \nu(\alpha_{M_1})\rightarrow M_2\setminus \nu(\alpha_{M_2}).
\end{equation*}
If $p > 0$ is an odd integer, then the four-manifolds $M_1(p)$ and $M_2(p)$ constructed in (\ref{Manifolds Produced}) are homeomorphic.
\end{theorem}

\begin{remark}\label{Remark Hypothesis} Theorem \ref{homeo} is false if $p > 0$ is an even integer. Indeed, Pao showed in \cite{pao}
that the $\Q$-homology four-spheres with $\pi_1 = \Z/p\Z$
\begin{equation}\label{spin qsphere}
L_{p} = S_{p} \cup_{f} (D^2\times S^2)
\end{equation}and \begin{equation}\label{nonspin qsphere}
L'_{p} = S_{p} \cup_{f'} (D^2\times S^2)
\end{equation}are not homotopy equivalent provided $p > 0$ is an even integer, and $f$ and $f'$ are non-isotopic self-diffeomorphisms of $S^1\times S^2$. 
If $p$ is odd, Pao also showed that the four-manifolds (\ref{spin qsphere}) and (\ref{nonspin qsphere}) are diffeomorphic. A handlebody of $L_3$ is drawn in Figure 1 for the sake of clarity. It is immediate to extend it to any value of $p > 1$ by wrapping the middle 2-handle $p$ times around the 1-handle and obtain handlebodies of $L_p$ and $L_p'$ according to the parity of the framing $l$ of the 2-handle that links the 0-framed 2-handle once; cf. \cite[Figure 5.46]{gompf19994}.
\end{remark}

\begin{remark} Gluck twists and the result of Pao. It follows from (\ref{spin qsphere}) and (\ref{nonspin qsphere} that $L_p'$ is obtained by performing a Gluck twist to $L_p$ \cite[p. 156]{gompf19994}, \cite[Definition 2.1]{akbulutyasui}. This is visible in the handlebody of Figure 1 by chosing the framing $l$ of the 2-handle $l\in \{-1, 0\}$ \cite[Figure 1]{akbulutyasui}. Another proof of Pao's result that this Gluck twist does not alter the diffeomorphism class if $p>1$ is odd is obtained by introducing a cancelling handle pair of a 1-handle and a 1-framed 2-handle \cite[Figure 5.12]{gompf19994} and performing a series of handleslides; cf. \cite[Figures 5.10, 5.11, 5.35, 5.39, 5.43 and 12.19]{gompf19994}. 

\end{remark}

\begin{remark}\label{Remark Handlebody}A handlebody of $S_p$. Remove the 3- and 4-handles, along with the Hopf link formed by the $l$-framed and $0$-framed 2-handles from Figure 1 to obtain a handlebody of $S_3$. A handlebody for any $p > 0$ is obtained by wrapping the middle 2-handle $p$-times around the horizontal dotted circle.
\end{remark}

\begin{figure}[ht]\label{Figure 1}
	\includegraphics[scale=0.5]{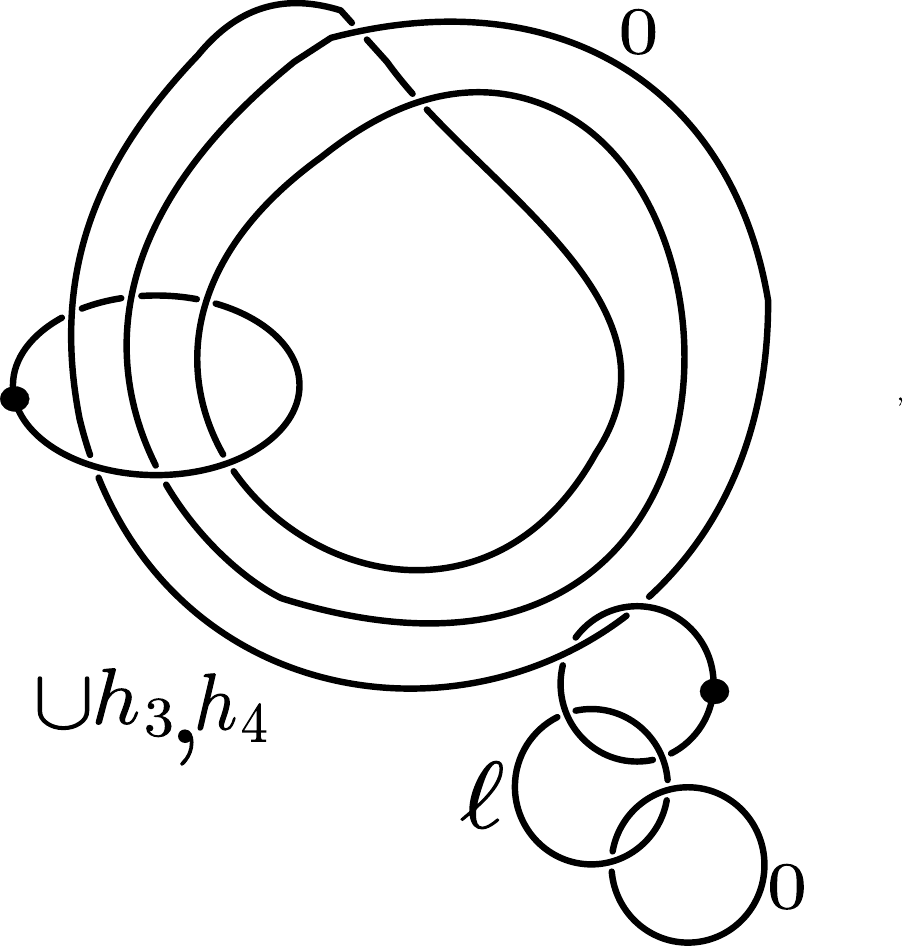}
	\caption{The diagram of the rational homology
four-sphere $L_p$. The middle curve wraps around $p$ times}
	\label{fig:KirbyDiagram}

\end{figure}

\subsection{Construction of the four-manifolds of Theorem \ref{mainthm}}\label{Section Construction Examples}To build our examples, we plug in the four-manifolds of Theorem \ref{bb} into the mechanism described in Section \ref{Section Construction Mechanism}. Fix a knot $K_m\subset \mathcal{K}$ (see (\ref{Family Knots})) and consider a simple framed loop $\alpha_{Y_{K_m}}\subset Y_{K_m}$ such that its homotopy class $[\alpha_{Y_{K_m}}]\in\pi_1(Y_{K_m}) = \Z/2\Z\times \Z/2\Z$ is a generator of the group, and remove its tubular neighborhood. This yields a compact four-manifold $Y_{K_m}\setminus\nu(\alpha_{Y_{K_m}})$ with boundary $\partial (Y_{K_m}\setminus\nu(\alpha_{Y_{K_m}})) = S^1\times S^2$, which will be the four-manifolds in the first step of the cut-and-paste procedure of Section \ref{Section Construction Mechanism}. 

Carrying the second and third steps yields a four-manifold\begin{equation}\label{New Manifolds}X_{K_m}(p):= (Y_{K_m}\setminus \nu(\alpha_{Y_{K_m}}))\cup S_p\end{equation}for every odd $p > 0$ and $S_p$ the compact four-manifold (\ref{Building Block Sp}). As observed in \cite{bais}, the assemblage (\ref{New Manifolds}) carries along the assemblage of an associated double cover of $X_{K_m}(p)$; cf. \cite[Lemma 8]{bais}. The necessary info is encoded in the next diagram.

\begin{proposition}\label{Proposition Coverings}There is a diagram\[\begin{tikzcd}
	& {K3_{K}\# L_p \# L_p} \\
	{E_{K}\# L_p} \\
	& {X_{K}(p)}.
	\arrow["{2:1}"', from=1-2, to=2-1]
	\arrow["{4:1}", from=1-2, to=3-2]
	\arrow["{2:1}"', from=2-1, to=3-2]
\end{tikzcd}\]of covers of the four-manifold assembled in (\ref{New Manifolds}).\end{proposition}

\begin{proof}Let $\widetilde{X_{K_m}(p)}\rightarrow X_{K_m}(p)$ be the double cover that corresponds to the subgroup $\langle \alpha,\beta^2\rangle$ in $\pi_1(X_{K_m}(p))$, the kernel of the group homomorphism $\varphi: \pi_1 (X(p))\to \Z/2\Z$ that sends $\alpha$ to $0$ and $\beta$ to $1$ according to the group presentation of \ref{presentazione}. Let $\widetilde{\alpha_{Y_{K_m}}}\subset E_{K_m}$ be the preimage of the loop $\alpha_{Y_{K_m}}\subset Y_{K_m}$ used in the assemblage (\ref{New Manifolds}) of $X_{K_m}(p)$. The former is a nullhomotopic simple loop in $E_{K_m}$. Since any two such loops in a four-manifold are isotopic \cite[Example 4.1.3]{gompf19994}, we have that\begin{equation}\label{Left Double Cover}\widetilde{X_{K_m}(p)} = E_{K_m}\setminus \nu(\widetilde{\alpha_{Y_{K_m}}})\cup S_p = E_{K_m}\#((D^2\times S^2)\cup S_p).\end{equation} As discussed in Remark \ref{Remark Hypothesis}, a result of Pao \cite{pao} implies that the piece $(D^2\times S^2)\cup S_p$ on the right-hand side of the connected sum in (\ref{Left Double Cover}) is diffeomorphic to the $\Q$-homology four-sphere $L_p$ since $p > 1$ is assumed to be an odd integer. It follows that $\widetilde{X_{K_m}(p)}$ is diffeomorphic to the connected sum $E_{K_m}\# L_p$. This establishes the double cover on the left-hand side of the diagram.  

The double cover $K3_{K_m}\# L_p\#L_p\rightarrow E_{K_m}\# L_p$ is constructed from the cover on the left-hand side of the diagram along with the double cover $K3_{K_m}\rightarrow E_{K_m}$ that was discussed in Section \ref{Section Raw Materials}. This, in turn, implies that there is a four-to-one cover $K3_{K_m}\#L_p\#L_p\rightarrow X_{K_m}(p)$. This completes the diagram.\end{proof}

\subsection{Topological invariants and homeomorphism class}\label{top}We begin this section by computing the topological invariants of (\ref{New Manifolds}) and then determine its  homeomorphism class as in the statement of Theorem \ref{mainthm}. 

    \begin{proposition}\label{presentazione}
    Let $p>1$ be an odd integer. The four-manifold $X_{K_m}(p)$ constructed in (\ref{New Manifolds}) satisfies the following properties. 
    \begin{enumerate}
        \item Its fundamental group is infinite and its abelianization is given by $\text{Ab}(\pi_1(X_{K_m}(p)))\cong \Z/2p\Z\times\Z/2\Z$.
        \item Its Euler characteristic and its signature are\begin{center}$\chi(X_{K_m}(p))=6$ and $\sigma(X_{K_m}(p)) = -4$,\end{center} and its second Stiefel-Whitney class is $w_2(X_{K_m}(p))\neq0.$
        \item The intersection form over $\Z$ of $X_{K_m}(p)$ is negative definite
        \end{enumerate}
\end{proposition}

\begin{proof} We work with the group presentation\begin{equation}\label{Group Presentation}\Z/2\Z\times\Z/2\Z=\langle \alpha,\gamma\;|\;\alpha^2=\gamma^2=1=(\alpha\gamma)^2\rangle\end{equation}for the Klein four-group.  To prove the first item, apply the Seifert-van Kampen theorem to (\ref{New Manifolds}) and obtain the group presentation\begin{align*}\label{presentazione}
        \pi_1(X_{K_m}(p))&=\langle \alpha,\gamma,\beta \;|\; \alpha^2=\gamma^2=1=(\alpha\gamma)^2;\;\beta^p\alpha=1\rangle =\\
        &=\langle\gamma,\beta\;|\;\gamma^2=\beta^{2p}=1;\;\beta^p\gamma=\gamma\beta^p\rangle.
    \end{align*}
    Here, the element $\beta^p$ is in the center of $\pi_1(X_{K_m}(p))$. Let $N=\langle \beta^p\rangle$ be the normal subgroup generated by $\beta^p.$ We have that\begin{equation*}\pi_1(X_{K_m}(p))/N=\Z/2\Z*\Z/p\Z.\end{equation*}Since the fundamental group $\pi_1(X_{K_m}(p))$ has an infinite quotient, it is itself an infinite group. We observe that when considering the abelianization $\text{Ab}(\pi_1(X_{K_m}(p)))$, the presentation of the group reduces to $$\langle\gamma,\beta\;|\;\gamma^2=\beta^{2p}=1;\;\beta\gamma=\gamma\beta\rangle=\Z/2p\Z\times\Z/2\Z.$$The values in the second item of the proposition are obtained from straight-forward (co)homological computations, Novikov additivity and Theorem \ref{bb}, and one sees that\begin{center}$\chi(X_{K_m}(p)) = \chi(Y_{K_m}) = 6$, $b_2(X_{K_m}(p)) = b_2(Y_{K_m}) = 4$ and $\sigma(X_{K_m}(p)) = \sigma(Y_{K_m}) = - 4$.\end{center}Rokhlin’s theorem implies $w_2(X_{K_m}(p))\neq 0$. The last item follows from the equality $|\sigma(X_{K_m}(p))| = b_2(X_{K_m}(p))$.
\end{proof}

    \begin{remark}\label{Remark Growth}
        The presentation of $\pi_1(X_m(p))$ that is described in Proposition \ref{presentazione} allows us to conclude that these fundamental groups grow exponentially. Indeed, the subgroup $\langle\beta^p\rangle<\pi_1(X_m(p))$ is normal and finite. Hence, $\pi_1(X_m(p))$ is quasi-isometric to $\pi_1(X_m(p))/\langle\beta^p\rangle\cong\Z/2\Z*\Z/p\Z$ \cite[Corollary 8.47]{dructu2018geometric}. The latter has exponential growth, as it possesses a subgroup isomorphic to the free group on two generators $F_2$ provided $p>2$.
    \end{remark}

A straightforward application of Theorem \ref{homeo} to our setting allows us to pin the homeomorphism class of the four-manifolds under consideration.
\begin{proposition}\label{Proposition Homeomorphism}The set $\{X_{K_m}(p): K_m\in \mathcal{K}\}$ consists of pairwise homeomorphic four-manifolds for $p > 1$ an odd integer.
\end{proposition}

\begin{proof} Let $Y_{K_{m_1}}$ and $Y_{K_{m_2}}$ be any two four-manifolds in the infinite family (\ref{Family HNP}) of Theorem \ref{bb} for $K_{m_1}, K_{m_2}\in \mathcal{K}$. Theorem \ref{bb} says that there is a homeomorphism\begin{equation}\label{Homeomorphism 1}h:Y_{K_{m_1}}\rightarrow Y_{K_{m_2}}.\end{equation}The homeomorphism (\ref{Homeomorphism 1}) canonically restricts to a homeomorphism\begin{equation}\label{Homeomorphism 2}h|:Y_{K_{m_1}}\setminus \nu(\alpha_{Y_{K_{m_1}}})\rightarrow Y_{K_{m_2}}\setminus \nu(h(\alpha_{Y_{K_{m_1}}})),\end{equation}where the homotopy classes\begin{center}$[\alpha_{Y_{K_{m_1}}}]\in \pi_1(Y_{K_{m_1}})$ and $[h(\alpha_{Y_{K_{m_1}}})]\in \pi_1(Y_{K_{m_2}})$\end{center}are generators. Proceed as in (\ref{New Manifolds}) and cap off the boundaries of these four-manifolds with $S_p$ to produce\begin{equation}\label{New Manifold 1}X_{K_{m_1}}(p): =Y_{K_{m_1}}\setminus \nu(\alpha_{Y_{K_{m_1}}})\cup S_p\end{equation}and\begin{equation}\label{New Manifold 2}X_{K_{m_2}}(p): = Y_{K_{m_2}}\setminus \nu(h(\alpha_{Y_{K_{m_1}}}))\cup S_p,\end{equation}which are elements of the set $\{X_{K_m}(p): K_m\in \mathcal{K}\}$. The four-manifolds (\ref{New Manifold 1}) and (\ref{New Manifold 2}) are homeomorphic by (\ref{Homeomorphism 2}) and Theorem \ref{homeo}. Since the choice of initial four-manifolds was arbitrary, the proposition follows.\end{proof}

\section{Obstructing diffeomorphisms.}\label{sw}
\noindent We follow the strategy in \cite{bais, baykur2024smooth, harris, levine2023new, stipsicz2024definite, stipsicz2024definite2} to distinguish the smooth structures on our four-manifolds through the smooth structures on their finite degree coverings. Recall that a diffeomorphism $f:M_1\rightarrow M_2$ between two closed oriented four-manifolds lifts to a diffeomorphism $\tilde{f}:\widetilde{M}_1\rightarrow \widetilde{M}_2$, whenever the subgroups of $\pi_1(M_i)$ that identify the coverings correspond under the group homomorphism induced by $f$. 

\begin{lemma}\label{lift}
    Let\begin{equation}\label{Diffeo Down}f:X_{K_m}(p)\rightarrow X_{K_n}(p)\end{equation}be a diffeomorphism with two respective given coverings $K3_{K_m}\#L_p\#L_p$ and $K3_{K_n}\#L_p\#L_p$ associated as in Proposition \ref{Proposition Coverings}. The diffeomorphism (\ref{Diffeo Down}) lifts to a diffeomorphism\begin{equation}\label{Diffeo Up}\tilde{f}: K3_{K_m}\#L_p\#L_p\rightarrow K3_{K_n}\#L_p\#L_p.\end{equation}
\end{lemma}We obstruct the existence of the diffeomorphism (\ref{Diffeo Down}) by obstructing the existence of the diffeomorphism (\ref{Diffeo Up}) with the Seiberg-Witten invariants through the computations of Harris-Naylor-Park \cite{harris}. More precisely, we use the diagram in Proposition \ref{Proposition Coverings} and make the following observation regarding the Seiberg-Witten invariants of the connected sum $K3_{K_m}\#L_p\#L_p$; see Section \ref{Section SW} regarding the notation used.

\begin{proposition}\label{Proposition CS}
\begin{equation}\label{Identity Proposition CS}\overline{\Sw}_{K3_{K_m}\#L_p\#L_p} = \overline{\Sw}_{K3_{K_m}}.\end{equation}
\end{proposition}

\begin{proof}Since the connected sum $L_p\# L_p$ is a $\Q$-homology four-sphere with $b_1(L_p\#L_p) = 0 = b_2(L_p\# L_p)$, it follows from \cite[Proof of Proposition 2]{Kotschick1995FourmanifoldsWS}, \cite[Corollary 2.11]{baykur2024smooth} that the identity\begin{equation}\label{SW Equality}\Sw_{K3_{K_m}}(\mathfrak{s}) = \pm \Sw_{K3_{K_m}\#L_p\#L_p}(\mathfrak{s}^{\#})\end{equation}holds for any extension $(K3_{K_m}\#L_p\#L_p, \mathfrak{s}^\#)\in \textnormal{Spin}^\C(K3_{K_m}\#L_p\#L_p)$ of the spin$^\C$-structure $(K3_{K_m}, \mathfrak{s})\in \textnormal{Spin}^\C(K3_{K_m})$. The identity (\ref{Identity Proposition CS}) follows from (\ref{SW Equality}) and the definition of the invariant $\overline{\Sw}$ discussed in Section \ref{Section SW}.\end{proof}

The main result of this section is the existence of infinitely many non-diffeomorphic irreducible smooth structures.

\begin{proposition}\label{Proposition Diffeomorphism}The set $\{X_{K_m}(p): K_m\in \mathcal{K}\}$ consists of infinitely many pairwise non-diffeomorphic irreducible four-manifolds, where $p > 1$ is an odd integer.
\end{proposition}

\begin{proof} The claim that the set consists of infinitely many pairwise non-diffeomorphic four-manifolds follows from Proposition \ref{Proposition CS}, Lemma \ref{lift} and the Seiberg-Witten invariant computations of Theorem \ref{bb}. We proceed to prove irreducibility. Fix an odd $p > 0$ and a knot $K\in \mathcal{K}$, and set $X_p:= X_{K_m}(p)$. We proceed by contradiction and assume that there is a smooth connected sum decomposition $X = M_1\# M_2$ with $M_i$ not a homotopy four-sphere for $i = 1, 2$. Notice that $\pi_1(X)$ cannot be expressed as a free product of non-trivial
groups since it has non-trivial center, while the free product of non-trivial
groups has trivial center. Without loss of generality, suppose $\pi_1(X) = \pi_1(M_1)$ and $\pi_1(M_2) = \{1\}$. We have that $b_2(M_2) > 0$ and the intersection form of $X$ splits as $Q_X  =Q_{M_1}\oplus Q_{M_2}$. We claim that $Q_{M_2}$ is an even form. Suppose that $Q_{M_2}= m\langle 1\rangle \oplus n\langle -1\rangle$ for some $m, n \geq 0$. Freedman's result implies that $M_2$ is homeomorphic to $m\mathbb{CP}^2\# n\overline{\mathbb{CP}^2}$ \cite[Theorem 1.2.27]{gompf19994}. If there were such a decomposition this would lift to the covering of $X$, which is impossible since the intersection form of the Enriques surface is\begin{equation*}-E_8\oplus \begin{pmatrix}
0& 1 \\
1 & 0 
\end{pmatrix}\end{equation*}for $-E_8$ a negative definite, even unimodular form of rank eight \cite[p. 13]{gompf19994}. This implies that $Q_{M_2}$ is an even form of rank four; its rank is not zero since we assumed that $M_2$ is not a homotopy four-sphere. However, this is impossible as well since the rank of an even form is divisible by eight \cite[Lemma 1.2.20]{gompf19994}. We conclude that $X_p$ is irreducible.\end{proof}

\section{Proof of Theorem \ref{mainthm}}\label{Section Proof Theorem Main}

The four-manifolds of our main result were constructed in (\ref{New Manifolds}). The homeomorphism class of these four-manifolds was established in Proposition \ref{Proposition Homeomorphism}. Since the choice of knot is immaterial in terms of the homeomorphism class obtained in the construction and to synchronize notation with the statement of the theorem, we fix $X_p:= X_{K_m}(p)$ for $K_m\in \mathcal{K}$; this notation was already used in the the proof of Proposition \ref{Proposition Diffeomorphism}. The values discussed in Item (3) were calculated in Proposition \ref{presentazione} to be $\chi(X_p) = 6$, $\sigma(X_p) = - 4$ and $w_2(X_p)\neq 0$. Since $b_2(X_p) = |\sigma(X_p)|$, it follows that the intersection form of $X_p$ is negative definite. This establishes Item (2). The existence of infinitely many non-diffeomorphic irreducible smooth structures was proven in Proposition \ref{Proposition Diffeomorphism}. This concludes the proof of our main result.\hfill $\square$


\begin{remark}\label{Remark Extension}
    It is interesting to understand what is obtained from applying the cut-and-paste mechanism of Section \ref{Section Construction Mechanism} to the four-manifold $X_p$ of Theorem \ref{mainthm} along a loop $\alpha_{X_p}\subset X_p$ whose homotopy class is a generator of order two. This yields a four-manifold $X_{p, p'} = X_p\setminus \nu(\alpha_{X_p})\cup S_{p'}$ for $p' > 0$ whose fundamental group abelianizes to $ \Z/2p\Z\times \Z/2p'\Z$. 
\end{remark}

\printbibliography[heading=bibintoc,title={References}]

\end{document}